\documentclass[12pt]{article}

\usepackage{amsfonts,graphicx,amsmath,amssymb,amsthm,geometry,mathtools,authblk,mathrsfs,bbm,dsfont,hyperref,nicefrac,enumerate,comment,cite,color}

\newcommand{\R}{\mathbb{R}}
\newcommand{\N}{\mathbb{N}}

\newcommand{\norm}[1]{\left \lVert#1\right \rVert}

\newtheorem{lemma}{Lemma}

\newtheorem{theorem}[lemma]{Theorem}

\begin{document}
	
	\title{Convergence results for gradient flow and gradient descent systems in the artificial neural network training}
	
	\author{}

	\author{Arzu Ahmadova
		\bigskip
		\\
		\small{Faculty of Mathematics, University of Duisburg-Essen, Essen,
			Germany,\\
			e-mail: \texttt{arzu.ahmadova@uni-due.de}\\
		}
		
		\smallskip
	}
	
	\maketitle
	\begin{abstract}
		The field of artificial neural network (ANN) training has garnered significant attention in recent years, with researchers exploring various mathematical techniques for optimizing the training process. In particular, this paper focuses on advancing the current understanding of gradient flow and gradient descent optimization methods. Our aim is to establish a solid mathematical convergence theory for continuous-time gradient flow equations and gradient descent processes based on mathematical anaylsis tools.
	\end{abstract}
	
	\tableofcontents

	\section{Introduction}
	Artificial neural networks (ANNs) have led to performance improvements in various tasks involving rectified linear unit (ReLU) activation via gradient flow (GF) and gradient descent (GD) schemes.	GF and GD systems are closely related concepts that are often used in optimization and machine learning. GF systems refer to the dynamics of a function evolving over time under the influence of its gradient. These systems can be thought of as a continuous version of gradient descent, where the parameters of the function change continuously rather than in discrete steps. GF systems are used in a variety of applications, including machine learning, physics, and chemistry. On the other hand, GD is an optimization algorithm aimed at minimizing a function. It works by iteratively adjusting the parameters of the function in the direction of the negative gradient, which is the direction of steepest decrease in the function's value. This process continues until the parameters reach a point where the gradient is very close to zero, indicating that a minimum has been found in the training of ANNs with ReLU activation function.
   Hence, GF represents the evolution of a function under its gradient, while GD is an algorithm for minimizing a function.
	
	Standard convergence results for GF and GD systems frequently rely on the convexity of the potential function near an isolated minimum. The convergence of GF and GD processes to the global minimum for convex objective functions has been established in various settings, as demonstrated in studies such as \cite{bach-moulines,jentzen-koerger,nesterov}. For more information on abstract convergence results for GF and GD processes in non-convex settings, please refer to the studies in \cite{BerTsit, GarrigosGower} and the references cited within.
	
	In the analysis of convergence of gradient descent scheme,  Lojasiewicz inequality has played important role. Note that the Lojasiewicz inequality implies that if $F: \R^{n} \to \R$ is a real analytic function, then any bounded solution $x$ of the gradient system converges to a critical point of $F$ as $t$ tends to infinity. 
	For a deeper understanding of the Lojasiewicz convergence theorem, one can refer to references \cite{L-1,L-2}. A revisited theorem of this convergence has been studied by Haraux in his article \cite{haraux2} and in his book \cite{haraux} (Chapter 7), in collaboration with Jendobi.
	
	In recent papers, researchers have been exploring various aspects of the convergence of these algorithms, such as the impact of the step size, the presence of noise, the choice of initialization, and the properties of the loss function. Some papers also compare the performance of the gradient flow and gradient descent systems under different conditions and with different types of ANNs. The main goal of these analyses is to understand how the parameters of the ANN are updated during the training process and how the training error decreases over time. 
	
	Although there are many scientific articles about the convergence analysis of GD, there are relatively fewer articles about the convergence analysis of GF processes in the context of training ANNs. Eberle et al. in \cite{eberletal} showed that the objective functions in the training of ANNs with ReLU activation meet the requirements for an appropriate Lojasiewicz inequality if the target function and the input data's probability distribution are piecewise polynomial. For convergence analyses of GF and GD processes with constant target functions, refer to \cite{ChJeRiRoss}. To understand convergence analysis of GF and GD processes in the training of ANNs with piecewise linear target functions, consult \cite{jentzen-riekert-2}.

	In this article, we first consider the time-continuous gradient system for all $t \in [0,\infty)$, $x\in \mathbb{C}^{1}([0,\infty),\R^{d})$ and $F \in \mathbb{C}^{1}(\R^{d},\R)$ that
	\begin{equation}\label{eq0}
	x^{\prime}(t)=-\left( \nabla F\right) (x(t)).
\end{equation}
	If $x$ is a solution to the gradient system \eqref{eq0}, its time derivative is always equal to the negative gradient $-\nabla F(x)$, which indicates the direction of steepest descent. As a result, it's reasonable to assume that every solution $x$ to \eqref{eq0} has the property that $F$ along the solution $x$ is non-increasing. When $x$ is a solution to \eqref{eq0} and $F$ is continuously differentiable, the composition $F\circ x$ is also non-increasing. If $F\circ x$ is constant, then $x$ itself is constant.

	It is improtant to determine whether $x(t)$ always converges for all $t\in [0,\infty)$. However, in $2$ dimensions, it has been shown that convergence may not occur even for a $\mathbb{C}^{\infty}$ potential $F$ - this was conjectured by Curry \cite{curry} and proven by Palis and de Melo \cite{PalisdeMelo}. The inequality $ \int_{0}^{\infty}|\left( \nabla F\right) (x(s))|\mathrm{d}s < \infty$ is therefore false for general gradient systems. If it were true, it would imply convergence, which has been shown to be false in the general smooth case. A counterexample to this was already exhibited by Curry in 1948, and has later been generalized in \cite[Section~17.1]{attouch-buttazo-michaille} and
	\cite[Section~10.3]{haraux}. As a consequence, the inequality $ \int_{0}^{\infty}\|\left( \nabla F\right) (x(s))\| ds < \infty $ is true when $F$ is analytic in a ball. This has been proven using Lojasiewicz gradient inequality. It is worth noting that in this section, the results are revisited by considering $F$ from $C^1$ rather than the $C^2$ condition established in \cite{haraux2}.
	
	In Section \ref{sec3}, we study convergence analysis of the following GD systems:
	\begin{equation}
		x_{n+1}=x_{n} -\gamma_{n} \left( \nabla F \right) (x_{n}).
	\end{equation}
In a discrete time setting, this paper makes a significant contribution by providing a convergence proof for the GD system. The key result establishes convergence using the inequality $ (F^{\prime}(x)-F^{\prime}(y))(x-y) \leq c|x-y|^{1+\alpha} $, but with weaker assumptions than previous works. Specifically, this section can be considered as a special case of the study conducted by Dereich and Kassing \cite{dereich-kassing}. However, the novel aspect lies in the consideration of a perturbed term, where it is assumed to be zero, even under weaker assumptions compared to the aforementioned paper. By demonstrating the convergence of the GD system under these relaxed conditions, the research presented in this paper expands the understanding of convergence properties in discrete time settings. This finding is of great importance, as it opens up new avenues for applying GD algorithms in various practical scenarios.

	 The motivation behind this article stems from the fact that the convergence of GF and GD processes is not well understood, particularly when it comes to GF processes. To enhance the accuracy and efficiency of the training process, it is crucial to deepen our understanding of the convergence behavior of these processes.
	 
	 This study aims to address this gap in the literature by focusing on the convergence analysis of both GF and GD processes in the training of ANNs. This research provides valuable insights and contributes to the field by improving the understanding of GF and GD processes and their convergence behaviors.
	 
    The remainder of this article is organized as follows. Section \ref{sec2} is devoted to proving the convergence results for GF systems with the help of Theorem \ref{Lem0} and Theorem \ref{Lem1}. In Section \ref{sec3}, we prove the convergence results for GD systems under weeker assumptions stated in Theorem \ref{Lem3} and Theorem \ref{Lem4}.

	\section{Convergence results for gradient flow systems} \label{sec2}
	  The following theorems are the main results of this section and establishes proofs of convergence results for GF scheme. We note that $F\in \mathbb{C}^{1}(\R^{d},\R)$ in this section, because the
	$ \mathbb{C}^{1}$ condition guarantees that the gradient of the function is differentiable.
	\begin{theorem}\label{Lem0}
		Let $d \in \mathbb{N}$, $x\in \mathbb{C}^{1}([0,\infty),\R^{d})$, $F\in \mathbb{C}^{1}(\R^{d},\R)$, and let $\left\| \cdot \right\| \colon \R^{d} \to [0, \infty)$ be a norm, assume $\inf_{t\in [0,\infty)}F(x(t))>-\infty$ and assume for all $t\in [0, \infty)$ that
		\begin{equation}\label{gradient system0}
      	x^{\prime}(t)= -\left( \nabla F\right) (x(t)).
		\end{equation}
		Then 
		\begin{description}
			\item[(i)] it holds that  $\big(F(x(t))\big)_{t\in [0,\infty)}$ converges in $\R$ as $t \to \infty$ and
			\item[(ii)] it holds that 
			\begin{equation}\label{eqF}
				\int_{0}^{\infty}\|\left( \nabla F\right) (x(s))\|^{2}\mathrm{d}s<\infty.
			\end{equation}	
		\end{description}   
	\end{theorem}
	\begin{proof} 
		(i) Note that equation \eqref{gradient system0} and the chain rule  assure for all $t \in [0,\infty)$ that
		\begin{equation}\label{chain rule}
			\frac{\mathrm{d}}{\mathrm{d}t}\Big((F\circ x)(t)\Big)=F^{\prime}( x(t)) x^{\prime}(t) = -\norm{\left( \nabla F\right) (x(t))}^2 \leq 0. 
		\end{equation}
		This and $\inf_{t\in [0,\infty)}F(x(t))>-\infty$ imply (i).\\\\
		(ii) Integrating \eqref{chain rule} ensures for all $t\in [0,\infty)$ that
		\begin{equation}\label{integral}
			\int_{0}^{t}\|\left( \nabla F\right) (x(s))\|^{2}\mathrm{d}s=F(x(0))-F(x(t)).
		\end{equation}
		This and $\inf_{t\in [0,\infty)}F(x(t))>-\infty$ imply (ii).
	\end{proof}	
	\begin{theorem}\label{Lem1}
		Let $d \in \mathbb{N}$,  $F\in \mathbb{C}^{1}(\R^{d},\R)$, let $x\in \mathbb{C}^{1}([0,\infty),\R^{d})$ be bounded, let $\left\| \cdot \right\|\colon \R^{d} \to [0, \infty)$ be a norm, and assume for all $t\in [0, \infty)$ that
		\begin{equation}\label{gradient system}
			x^{\prime}(t)= -\left( \nabla F\right) (x(t)).
		\end{equation}
		Then 
		\begin{description}
			\item[(i)] it holds that  $\big(F(x(t))\big)_{t\in [0,\infty)}$ converges in $\R$ as $t \to \infty$ and
			\item[(ii)] it holds that \begin{equation}\label{limsup}
				\limsup_{t\to \infty}\|\left( \nabla F\right) (x(t))\|=0.
			\end{equation}
		\end{description}
	\end{theorem}
	\begin{proof} Continuity of $F$ and boundedness of $x$ imply that $F \circ x$ is bounded. This and Theorem \ref{Lem0} imply (i) and
		\begin{equation}\label{eq4}
			\int_{0}^{\infty}\|\left( \nabla
			F\circ x\right) (s)\|^2 ds < \infty.
		\end{equation}
		The restriction of the continuous function $\nabla F$ to the compact set ${\overline{x([0,\infty))}}$ is bounded and uniformly continuous. This and equation \eqref{gradient system} imply that $x^{\prime}$ is bounded. The fact that $x$ has a bounded derivative guarantees that $x$ is uniformly continuous. Since the function $\nabla F$ is continuous and  $x$ is bounded and uniformly continuous, this yields that $\nabla F\circ x$ is uniformly continuous. This and the fact that the restriction of the continuous function $\left\| \cdot \right\|^{2}$ to the compact set $\overline{(\nabla F\circ x)([0,\infty))}$ is uniformly continuous, prove that $\|\nabla F\circ x\|^{2}$ is uniformly  continuous on $[0, \infty)$. Combining \cite[Theorem~ 2.1.2]{haraux}
		 together with uniform continuity of $\|(\nabla F)\circ x\|^{2}$ and \eqref{eq4} imply that 
		\begin{equation}
			\limsup_{t\to \infty}\|\left( \nabla F\right) (x(t))\|=0.
		\end{equation}
		This proves (ii) and completes the proof of Theorem \ref{Lem1}. 
	\end{proof}

	\section{Convergence results for gradient descent systems} \label{sec3}
	The following theorems are the main results in this section. We note that $\gamma_n$ is used as the step size in the $n$-th iteration of the gradient descent algorithm, which is defined by the equation $x_{n+1}=x_{n} -\gamma_{n}\left( \nabla F\right) (x_{n})$. The assumption $\overline{\lim}_{n\to \infty}\gamma_{n}=0$ ensures that the step size goes to zero as the algorithm progresses, which is necessary for the algorithm to converge to a minimum of the function. The assumption $\sum_{n=0}^{\infty}\mathds{1}_{(0,1)}(\alpha)\gamma_{n}^{\frac{1+\alpha}{1-\alpha}}<\infty$ ensures that the step size decays fast enough for the algorithm to converge. For consistency, $F$ is also selected from $\mathbb{C}^{1}(\R^{d},\R)$ in this section.
\begin{theorem}\label{Lem3}
	Let $d\in \N$, $c\in(0,\infty)$, $\alpha\in (0,1]$, $F\in \mathbb{C}^{1}(\R^{d},\R)$, let $\left\| \cdot \right\|\colon \R^{d} \to [0, \infty)$ be a norm, let $\gamma\colon \mathbb{N}_{0}\to (0,\infty)$ satisfy that $\overline{\lim}_{n\to \infty}\gamma_{n}=0$, let $C\subseteq \R^{d} $ be convex, assume for all $x,y\in C$ that
	\begin{equation}\label{eq5.1}
		(F^{\prime}(x)-F^{\prime}(y))(x-y) \leq c\|x-y\|^{1+\alpha},
	\end{equation}
	assume that $\sum_{n=0}^{\infty}\mathds{1}_{(0,1)}(\alpha)\gamma_{n}^{\frac{1+\alpha}{1-\alpha}}<\infty$, let $x \colon \N_{0}\to C$, assume $\inf_{n\in \N_{0}}F(x_{n})>-\infty$, and assume for all $n\in \N_{0}$ that
	\begin{equation}\label{eq6.1} 
		x_{n+1}=x_{n} -\gamma_{n}\nabla F(x_{n}).
	\end{equation}
	Then the following statements hold:
	\begin{description}
			\item[(i)] $(F(x_{n}))_{n\in \mathbb{N}_{0}}$ converges in $\R$ and
		\item[(ii)] $\sum_{n=0}^{\infty}\gamma_{n}\|\nabla F(x_{n})\|^{2}<\infty$.
	 
	\end{description}
\end{theorem}

\begin{proof}
	Note that the fundamental theorem of calculus yields for all $n\in \mathbb{N}_{0}$ that
	\begin{align}\label{FTC}
		&F(x_{n+1})-F(x_{n})\nonumber\\
		&=\int_{0}^{1}F^{\prime}(x_{n}+\lambda(x_{n+1}-x_{n}))(x_{n+1}-x_{n})\mathrm{d}\lambda \nonumber\\
		&=\int_{0}^{1}F^{\prime}(x_{n})(x_{n+1}-x_{n})\mathrm{d}\lambda\\
		&+\int_{0}^{1}\Big( F^{\prime}(x_{n}+\lambda(x_{n+1}-x_{n}))- F^{\prime}(x_{n})\Big)(x_{n}+\lambda(x_{n+1}-x_{n})-x_{n})\frac{1}{\lambda}\mathrm{d}\lambda.\nonumber
	\end{align}
Hence, equation \eqref{eq6.1} yields for all $n\in \N_{0}$ that
\begin{align}
	&F(x_{n+1})-F(x_{n})= -\gamma_{n}\|\nabla F(x_{n})\|^{2}\\	&+\int_{0}^{1}\Big( F^{\prime}(x_{n}+\lambda(x_{n+1}-x_{n}))- F^{\prime}(x_{n})\Big)(x_{n}+\lambda(x_{n+1}-x_{n})-x_{n})\frac{1}{\lambda}\mathrm{d}\lambda \nonumber.
\end{align} 
This, inequality \eqref{eq5.1}, and equation \eqref{eq6.1}  imply for all $n\in \N_{0}$ that
\begin{align}\label{**}
&F(x_{n+1})-F(x_{n})\nonumber\\	&\leq -\gamma_{n}\|\nabla F(x_{n})\|^{2}+c\int_{0}^{1}\lambda^{1+\alpha}\|x_{n+1}-x_{n}\|^{1+\alpha}\frac{1}{\lambda}\mathrm{d}\lambda\\&= -\gamma_{n}\|\nabla F(x_{n})\|^{2}+\frac{c}{1+\alpha}\gamma_{n}^{1+\alpha}\|\nabla F(x_{n})\|^{1+\alpha}\nonumber.
\end{align}
\textbf{Step 1:} Throughout Step 1 we assume $\alpha=1$. Then \eqref{**} implies for all $n\in \N_{0}$  that
	\begin{equation}\label{ineq12}
		F(x_{n+1})-F(x_{n})\leq -\|\nabla F(x_{n})\|^{2}\Big(\gamma_{n}-\frac{c}{2}\gamma_{n}^{2}\Big).
	\end{equation}
	This and $\overline{\lim}_{n\to \infty}\gamma_{n}=0$ ensure that $(F(x_{n}))_{n\in \N_{0}}$ is eventually monotonically non-increasing. This and $\inf_{n\in \N_{0}}F(x_{n})>-\infty$ imply that $\lim_{n\to \infty}F(x_{n})$ exists in $\R$. This proves (i) in the case $\alpha=1$.\\
	Summing over \eqref{ineq12} gives for all $n\in \N_{0}$ that	
	\begin{equation}\label{ineq13}
		F(x_{0})-F(x_{n+1})=-	\sum_{k=0}^{n}(F(x_{k+1})-F(x_{k}))\geq \sum_{k=0}^{n}\|\nabla F(x_{k})\|^{2}(\gamma_{k}-\frac{c}{2}\gamma_{k}^{2}).
	\end{equation}
	This, $\inf_{n\in \N_{0}}F(x_{n})>-\infty$ and $\overline{\lim}_{n\to \infty}\gamma_{n}=0$ imply that
	\begin{equation}
		\sum_{k=0}^{\infty}\gamma_{k}\|\nabla F(x_{k})\|^{2}<\infty.
	\end{equation}
	This proves (ii) in the case $\alpha=1$.\\
	\textbf{Step 2}: Throughout Step 2 we assume $\alpha<1$. Summing over \eqref{**} shows for all $n\in \N_{0}$ that
	\begin{equation}\label{ineq10}
		F(x_{n+1})-F(x_{0})\leq -\sum_{k=0}^{n}\gamma_{k}\|\nabla F(x_{k})\|^{2}+\frac{c}{1+\alpha}\sum_{k=0}^{n}\gamma_{k}^{\frac{1+\alpha}{2}}\|\nabla F(x_{k})\|^{1+\alpha}\gamma_{k}^{\frac{1+\alpha}{2}}.
	\end{equation}
	This and H\"{o}lder's inequality give for all $n \in \N_{0}$ that
	\begin{align}\label{ineq11}
		F(x_{n+1})-F(x_{0})&\leq -\sum_{k=0}^{n}\gamma_{k}\|\nabla F(x_{k})\|^{2}\\
		&+c\Big(\sum_{k=0}^{n}\gamma_{k}^{\frac{1+\alpha}{2}\frac{2}{1-\alpha}}\Big)^{\frac{1-\alpha}{2}}\Big(\sum_{k=0}^{n}\gamma_{k}\|\nabla F(x_{k})\|^{2}\Big)^{\frac{1+\alpha}{2}}\nonumber.
	\end{align}
	Aiming at a contradiction assume that $\sum_{k=0}^{\infty}\gamma_{k}\|\nabla F(x_{k})\|^{2}=\infty$. Then \eqref{ineq11},\\
	$\inf_{n\in \N_{0}}F(x_{n})>-\infty$ and $\sum_{k=0}^{\infty}\gamma_{k}^{\frac{1+\alpha}{1-\alpha}}<\infty$ ensure that
	\begin{align}
		\infty&=\Big(\sum_{k=0}^{\infty}\gamma_{k}\|\nabla F(x_{k})\|^{2}\Big)^{\frac{1+\alpha}{2}}\Big[\Big(\sum_{k=0}^{\infty}\gamma_{k}\|\nabla F(x_{k})\|^{2}\Big)^{\frac{1-\alpha}{2}}-c\Big(\sum_{k=0}^{\infty}\gamma_{k}^{\frac{1+\alpha}{1-\alpha}}\Big)^{\frac{1-\alpha}{2}}\Big]\nonumber\\
		&\leq F(x_{0})-\inf_{n\in \N}F(x_{n})<\infty.
	\end{align}
	This is a contradiction. This proves (ii) in the case $\alpha<1$. \\
		Next we prove (i). Summing over \eqref{**}, H\"{o}lder's inequality, and (ii) yield that 
	\begin{align}
		\sum_{n=0}^{\infty}\|F(x_{n+1})-F(x_{n})\|
		&\leq \sum_{n=0}^{\infty}\gamma_{n}\|\nabla F(x_{n})\|^{2}\nonumber\\
		&+\frac{c}{1+\alpha}\sum_{n=0}^{\infty}\gamma_{n}^{1+\alpha}\|\nabla
		F(x_{n})\|^{1+\alpha}\\
		&\leq \sum_{n=0}^{\infty}\gamma_{n}\|\nabla F(x_{n})\|^{2}\nonumber\\
		&+\frac{c}{1+\alpha}\Big(\sum_{n=0}^{\infty}\gamma_{n}^{\frac{1+\alpha}{2}\frac{2}{1-\alpha}}\Big)^{\frac{1-\alpha}{2}}\Big(\sum_{n=0}^{\infty}\gamma_{n}\|\nabla F(x_{n})\|^{2}\Big)^{\frac{1+\alpha}{2}}<\infty\nonumber.
	\end{align}
	This implies that $\big(F(x_{n})\big)_{n\in \N_{0}}$ is a Cauchy sequence, and hence, convergent in $\R$. This proves (i) in the case $\alpha<1$. The proof of Theorem \ref{Lem3} is thus completed.
\end{proof}

\begin{theorem}\label{Lem4}
	Let $d\in \N$, $c\in(0,\infty)$, $\alpha\in (0,1]$, $F\in \mathbb{C}^{1}(\R^{d},\R)$, let $\left\| \cdot \right\|\colon \R^{d} \to [0, \infty)$ be a norm, let $C\subseteq \R^{d}$ be a bounded and convex set, and let $\gamma\colon \mathbb{N}_{0}\to (0,\infty)$ satisfy that $\overline{\lim}_{n\to \infty}\gamma_{n}=0$, assume for all $x,y\in C$ that
	\begin{equation}\label{ineq5}
		(F^{\prime}(x)-F^{\prime}(y))(x-y) \leq c\|x-y\|^{1+\alpha},
	\end{equation}
	assume that $\sum_{n=0}^{\infty}\mathds{1}_{(0,1)}(\alpha)\gamma_{n}^{\frac{1+\alpha}{1-\alpha}}<\infty$, let $x\colon \N_{0}\to C$, and assume for all $n\in \N_{0}$ that
	\begin{equation}\label{eq6} 
		x_{n+1}=x_{n} -\gamma_{n}\nabla F(x_{n}).
	\end{equation}
	Then the following statements hold:
	\begin{description}
		\item[(i)] $(F(x_{n}))_{n\in \mathbb{N}_{0}}$ converges in $\R$,
		\item[(ii)] $\sum_{n=0}^{\infty}\gamma_{n}\|\nabla F(x_{n})\|^{2}<\infty$, and
		\item[(iii)] if  $\sum_{n=0}^{\infty}\gamma_{n}=\infty$,  then 
		\begin{equation}\label{lim}
			\overline{\lim}_{n\to \infty}\|\nabla F(x_{n})\|=0.
		\end{equation}
	\end{description}
\end{theorem}

\begin{proof}
	Continuity of $F$ and boundedness of $x$ ensure that $(F(x_{n}))_{n\in \N_{0}}$ is bounded. This and Theorem \ref{Lem3} imply (i) and (ii).\\
	(iii) Aiming at a contradiction assume that there exist $\varepsilon\in (0,\infty)$ and a subsequence $(n_{k})_{k\in \N_{0}}\subseteq \N$ such that $\underline{\lim\limits}_{k\to \infty}n_k=\infty$ and
	\begin{equation}\label{nablaF>e}
		\|\nabla F(x_{n_{k}})\|\geq  \varepsilon.
	\end{equation}
	Define $\kappa\in [0, \infty]$ by $\kappa\coloneqq \sup_{n\in \mathbb{N}_{0}}\|\nabla F(x_{n})\|$. Continuity of $\nabla F$ and boundedness of $x$ imply the boundedness of $(\left( \nabla F\right) (x_{n}))_{n\in \N_{0}}$ and thus $\kappa< \infty$. Without loss of generality we assume that $\kappa >0$. Since $\nabla F\big|_{\overline{C}}$ is uniformly continuous, then there exists $\delta\in (0,\infty)$ such that for all $x,y \in C$ with $\|x-y\|\leq \delta$ it holds that
	
	\begin{equation}\label{e/2}
		\|\nabla F(x)-\nabla F(y)\|\leq \frac{\varepsilon}{2}.
	\end{equation}
We assume without loss of generality that $\sup_{k\in \N_{0}}\gamma_{n_k}< \frac{\delta}{\kappa}$.
	The fact that $\sum_{l=0}^{\infty}\gamma_{l}=\infty$ implies that there exist $(m_{k})_{k\in \N_{0}}\subseteq \N$ which satisfy for all $k\in \N_{0}$ that $m_{k}\geq n_{k}$ and it holds that
	\begin{equation}\label{ineq8}
		\sum_{l=n_{k}}^{m_{k}}\gamma_{l}\leq \frac{\delta}{\kappa}< \sum_{l=n_{k}}^{m_{k}+1}\gamma_{l}.
	\end{equation}
	Assume without loss of generality for all $k\in \N_{0}$ that $m_{k}+1<n_{k+1}$. Now for all $k\in \N_{0}$, $l\in \left\lbrace n_{k},\ldots, m_{k}+1 \right\rbrace $ we obtain from \eqref{eq6} and \eqref{ineq8} that
	\begin{equation}
		\|x_{l}-x_{n_{k}}\|= \norm{\sum_{i=n_{k}}^{l-1}(x_{i+1}-x_{i})}\leq  \sum_{i=n_{k}}^{l-1}\gamma_{i}\|\nabla F(x_{i})\|\leq \kappa \sum_{i=n_{k}}^{m_{k}}\gamma_{i}\leq \delta.
	\end{equation}
	This, \eqref{e/2} and \eqref{nablaF>e} imply for all $k\in \N_{0}$,  $l\in \left\lbrace n_{k},\ldots, m_{k}+1 \right\rbrace $ that 
	\begin{equation}\label{*}
		\|\nabla F(x_{l})\|\geq  \|\nabla F(x_{n_{k}})\|-\frac{\varepsilon}{2}\geq  \varepsilon-\frac{\varepsilon}{2}=\frac{\varepsilon}{2}.
	\end{equation}
	Furthermore, this together with \eqref{ineq8} ensures for all $k\in \N_{0}$ that
	\begin{equation}
		\sum_{l=n_{k}}^{m_{k}+1}\gamma_{l}\|\nabla F(x_{l})\|^{2}\geq   \Big(\frac{\varepsilon}{2}\Big)^{2}\sum_{l=n_{k}}^{m_{k}+1}\gamma_{l}\geq \Big(\frac{\varepsilon}{2}\Big)^{2}\frac{\delta}{\kappa}.
	\end{equation}
	Next this implies that
	\begin{equation}
		\sum_{l=0}^{\infty}\gamma_{l}\|\nabla F(x_{l})\|^{2}\geq  \sum_{k=0}^{\infty}\sum_{l=n_{k}}^{m_{k}+1}\gamma_{l}\|\nabla F(x_{l})\|^{2}\geq  \sum_{k=1}^{\infty} \Big(\frac{\varepsilon}{2}\Big)^{2}\frac{\delta}{\kappa}=\infty.
	\end{equation}
	This is a contradiction to (ii). The proof of Theorem \ref{Lem4} is thus completed.
\end{proof}
	\section*{Acknowledgement} 
	I would like to express sincere gratitude to Prof. Dr. Martin Hutzenthaler for his guidance and meticulous review of the paper. I would also like to thank Prof. Dr. Alain Haraux for his assistance in comprehending the theory of gradient flow systems and patiently addressing all questions pertaining to convergence analysis.\\
	This work has been funded by the
	Deutsche Forschungsgemeinschaft (DFG, German Research Foundation) through the research grant number HU1889/7-1.

\end{document}